\documentclass[11pt]{amsart}
\usepackage{amssymb,amsthm,amsmath,amstext,tikz-cd}
\usepackage{mathrsfs}  
\usepackage{bm}        
\usepackage{mathtools} 
\usepackage{changepage}
\usepackage{color,fullpage}

\usepackage[all]{xy}
 \DeclareFontFamily{U}{wncy}{}
    \DeclareFontShape{U}{wncy}{m}{n}{<->wncyr10}{}
    \DeclareSymbolFont{mcy}{U}{wncy}{m}{n}
    \DeclareMathSymbol{\Sha}{\mathord}{mcy}{"58} 

\theoremstyle{plain}
\newtheorem{theorem}{Theorem}[section]
\newtheorem{prop}[theorem]{Proposition}
\newtheorem{lemma}[theorem]{Lemma}
\newtheorem{cor}[theorem]{Corollary}

\newtheorem*{theoremintro*}{Theorem}

\theoremstyle{definition}

\newtheorem{example}[theorem]{Example}
\newtheorem{examples}[theorem]{Examples}

\theoremstyle{remark}
\newtheorem{remark}[theorem]{Remark}
\newtheorem{remarks}[theorem]{Remarks}

\newtheorem*{ack}{Acknowledgements}

\newcommand{\F}{\mathbb F}
\newcommand{\R}{\mathbb R}
\renewcommand{\P}{\mathbb P}
\newcommand{\Q}{\mathbb Q}

\newcommand{\Or}{\text{O}}
\newcommand{\SO}{\text{SO}}
\newcommand{\SHP}{\text{Hasse principle for isotropy}}
\newcommand{\WHP}{\text{Hasse principle for isometry}}

\newcommand{\Char}{\text{char }}

\usepackage[backref=page]{hyperref}
\hypersetup{pdftitle={Hasse principles for quadratic forms over function fields}}
\hypersetup{pdfauthor={Connor Cassady}}
\hypersetup{colorlinks=true,linkcolor=red,anchorcolor=blue,citecolor=blue}

\title[Hasse principles for quadratic forms]{Hasse principles for quadratic forms over function fields}
\date{\today}
\author[Cassady]{Connor Cassady}
\thanks{
The author was supported by NSF grants DMS-1805439 and DMS-2102987.}

\begin{document}

\maketitle
\vspace{-0.5cm}
\begin{abstract}
We investigate the Hasse principles for isotropy and isometry of quadratic forms over finitely generated field extensions with respect to various sets of discrete valuations. Over purely transcendental field extensions of fields that satisfy property $\mathscr{A}_i(2)$ for some $i$, we find numerous counterexamples to the Hasse principle for isotropy with respect to a relatively small set of discrete valuations. For finitely generated field extensions $K$ of transcendence degree $r$ over an algebraically closed field of characteristic $\ne 2$, we use the $2^r$-dimensional counterexample to the Hasse principle for isotropy due to Auel and Suresh to obtain counterexamples of lower dimensions with respect to the divisorial discrete valuations induced by a variety with function field $K$. 
\end{abstract}

\section*{Introduction}
The Hasse-Minkowski Theorem states that a quadratic form defined over a global field is isotropic if and only if it is isotropic over all completions of the field, and is one of the first examples of a \textit{local-global principle} for quadratic forms. This local-global principle for isotropy implies the local-global principle for isometry of quadratic forms over global fields. That is, two quadratic forms over a global field are isometric if and only if they are isometric over all completions of the field. For other fields, these local-global principles can be phrased in terms of discrete valuations on the field. 

For any field $k$ of characteristic $\ne 2$, let $V$ be a non-empty set of non-trivial discrete valuations on $k$. We say that a quadratic form $q$ over $k$ satisfies the \textit{Hasse principle for isotropy with respect to ~$V$} if $q$ being isotropic over the $v$-adic completion $k_v$ for all $v \in V$ implies that $q$ is isotropic over $k$. Given a pair of quadratic forms $q_1$ and $q_2$ over $k$, we say that $q_1$ and $q_2$ satisfy the \textit{Hasse principle for isometry with respect to $V$} if $q_1$ and $q_2$ being isometric over $k_v$ for all $v \in V$ implies that $q_1$ and ~$q_2$ are isometric over $k$. Two quadratic forms $q_1$ and $q_2$ are isometric if and only if they have the same dimension and $q_1 \perp -q_2$ is hyperbolic, so the $\WHP$ with respect to ~$V$ is satisfied if and only if every even-dimensional quadratic form over $k$ that is hyperbolic over ~$k_v$ for all $v \in V$ is also hyperbolic over $k$.

In this article, we investigate whether or not these Hasse principles hold with respect to various sets of discrete valuations on function fields of characteristic $\ne 2$. The main result is the following (see Section \ref{main results} for terminology):
\begin{theoremintro*}[\ref{main theorem}]
Let $k$ be an algebraically closed field of characteristic $\ne 2$ that is not the algebraic closure of a finite field. Let $K$ be any finitely generated field extension of transcendence degree $r \geq 2$ over $k$, and let $V$ be any non-empty set of non-trivial divisorial discrete valuations on $K$ that satisfies the finite support property. Then for any integer $m \ne 3$ such that
\[
	2^{r-1} < m \leq 2^r,
\]
there exists an $m$-dimensional quadratic form over $K$ that violates the Hasse principle for isotropy with respect to $V$.
\end{theoremintro*}
Previously, Auel and Suresh \cite{auel-suresh} showed that if $K$ is any finitely generated field extension of transcendence degree $r \geq 2$ over an algebraically closed field $k$ of characteristic $\ne 2$ that is not the algebraic closure of a finite field, then there is a quadratic form over $K$ of dimension $2^r$ that violates the Hasse principle for isotropy with respect to the set of all discrete valuations on $K$. This field has $u$-invariant $2^r$, so this dimension is a natural first place to look for counterexamples to the Hasse principle for isotropy, as any quadratic form of dimension $> 2^r$ is isotropic over $K$, thus automatically satisfies the $\SHP$. Theorem \ref{main theorem} partially generalizes \cite[Theorem ~1]{auel-suresh} by finding counterexamples of dimension $< 2^r$, but not with respect to the set of all discrete valuations on $K$. An example of a set of discrete valuations on $K/k$ to which Theorem ~\ref{main theorem} applies is the set of discrete valuations on $K$ induced by prime divisors on a projective integral regular $k$-scheme with function field $K$. 

As one might expect, the smaller the set of discrete valuations is, the easier it is to violate these Hasse principles. However, over rational function fields, the $\WHP$ does in fact hold with respect to a small set of discrete valuations (relative to the set of all discrete valuations on the field; see Proposition ~\ref{WHP over one-variable rational function field}\ref{part a}). Despite that, there are numerous examples of quadratic forms over these fields that violate the $\SHP$ with respect to this same set of discrete valuations. Indeed, in Section ~\ref{transcendence degree 1}, we prove the following result (see Section \ref{notation} for notation and terminology):
\begin{theoremintro*}[\ref{SHP failure over A_i fields, weak assumptions}]
Let $\ell$ be a field of characteristic $\ne 2$. Assume $\ell \in \mathscr{A}_i(2)$ for some $i \geq 0$ and $u(\ell) = 2^i$. For any $r \geq 1$ let $L_r = \ell(x_1, \ldots, x_r)$, and for $r \geq 2$ let $V_r$ be the set of discrete valuations on $L_r$ that are trivial on $L_{r-1}$. Then for $r \geq 2$ and any integer $m \ne 3$ such that 
\[
	2^{i + r - 1} < m \leq 2^{i+r},
\]
there exists an $m$-dimensional quadratic form over $L_r$ that violates the Hasse principle for isotropy with respect to $V_r$.
\end{theoremintro*}

\section{Notation and Preliminaries}
\label{notation}
All of the fields we consider will have characteristic different from 2, and all quadratic forms (occasionally referred to just as forms) considered will be nondegenerate (or \textit{regular}). Our notation and terminology follows \cite{lam}, and we assume familiarity with basic notions of quadratic form theory (see \cite[Chapter I]{lam}). 

Let $q$ be an $n$-dimensional quadratic form over a field $k$. Since $\Char k \ne 2$, we can diagonalize $q$ over $k$, and write $q = \langle a_1, \ldots, a_n \rangle$, with each $a_i \in k^{\times}$. The set of elements of $k^{\times}$ represented by $q$ over $k$ will be denoted by $D_k(q)$. If $K/k$ is any field extension, $q_K$ will denote the quadratic form ~$q$ considered as a quadratic form over $K$. We write $q_1 \simeq q_2$ to denote that the quadratic forms $q_1$ and $q_2$ over $k$ are isometric, and we say that $q_1$ and $q_2$ are \textit{similar} if there exists some $a \in k^{\times}$ such that $q_1 \simeq a \cdot q_2 := \langle a \rangle \otimes q_2$. Given forms $q, \varphi$ over $k$, we say that $q$ is a \textit{subform} of $\varphi$, denoted $q \subseteq \varphi$, if there exists some quadratic form $\psi$ over $k$ such that $\varphi \simeq q \perp \psi$.

The \textit{hyperbolic plane} $\langle 1, -1 \rangle$ over $k$ will be denoted by $\mathbb{H}$, and we say that an even-dimensional quadratic form $q$ over $k$ is \textit{hyperbolic} if $q \simeq m\mathbb{H}$ for some positive integer $m$, where $m\mathbb{H}$ denotes the orthogonal sum of $m$ copies of $\mathbb{H}$. 

Given any $a_1, \ldots, a_n \in k^{\times}$, we let $\langle \langle a_1, \ldots, a_n \rangle \rangle$ denote the \textit{$n$-fold Pfister form}
\[
	\langle \langle a_1, \ldots, a_n \rangle \rangle = \langle 1, a_1 \rangle \otimes \cdots \otimes \langle 1, a_n \rangle.
\]
A Pfister form is isotropic if and only if it is hyperbolic \cite[Theorem~ X.1.7]{lam}. If $\varphi$ is a Pfister form over $k$, we can write $\varphi \simeq \langle 1 \rangle \perp \varphi'$, and $\varphi'$ is called the \textit{pure subform} of $\varphi$. We will use this notation for any quadratic form that represents 1, i.e., if $q \simeq \langle 1, a_2, \ldots, a_n \rangle$, then we let $q' := \langle a_2, \ldots, a_n \rangle$, which is well-defined up to isometry by Witt Cancellation \cite[Theorem ~I.4.2]{lam}.

A certain measure of the complexity of quadratic forms over a field $k$ is the \textit{$u$-invariant of $k$}, denoted by $u(k)$, which is defined as the maximal dimension of an anisotropic quadratic form over ~$k$ \cite[Definition ~XI.6.1]{lam}. If no such maximum exists, then $u(k) = \infty$ (e.g., $u(\R) = u(\Q) = \infty$). For fields considered in this paper, the $u$-invariant is known.

First, there are $C_i$ fields (see, e.g., \cite[Definition, ~\textsection 2]{pfister}). If $k$ is a $C_i$ field for some $i \geq 0$, then $u(k) \leq 2^i$. By \cite[Propositions ~2, 3]{pfister}, if $k$ is a $C_i$ field for some $i \geq 0$, and $K$ is a finitely generated field extension of transcendence degree $r$ over $k$, then $K$ is a $C_{i + r}$ field, thus $u(K) \leq 2^{i + r}$. 

More generally, there are non-$C_i$ fields for which the $u$-invariant is known, like $\Q_p$. For any prime ~$p$, $u(\Q_p) = 4$, but $\Q_p$ is not a $C_2$ field. However, $\Q_p$ does satisfy property $\mathscr{A}_2(2)$ considered in \cite{leep}, which we now recall. For any integer $i \geq 0$, a field $\ell$ of characteristic $\ne 2$ satisfies property $\mathscr{A}_i(2)$ (written $\ell \in \mathscr{A}_i(2)$) if every system of $s$ quadratic forms over $\ell$ in $n > s \cdot 2^i$ common variables has a nontrivial simultaneous zero in an extension field of $\ell$ of odd degree.

By \cite[Proposition ~1]{pfister}, if $k$ is a $C_i$ field, then $k \in \mathscr{A}_i(2)$. According to \cite[Proposition ~2.2]{leep}, if $\ell \in \mathscr{A}_i(2)$ for some $i \geq 0$, then $u(\ell) \leq 2^i$. Moreover, much like $C_i$ fields, if $\ell \in \mathscr{A}_i(2)$ for some $i \geq 0$, and $L$ is a finitely generated field extension of transcendence degree $r$ over $\ell$, then by \cite[Theorems ~2.3, 2.5]{leep}, $L \in \mathscr{A}_{i + r}(2)$. For such a field $L/\ell$, we conclude that $u(L) \leq 2^{i + r}$.

When discussing the $\SHP$ over a field $k$, it is beneficial to know how quadratic forms behave over the residue fields of the completions of $k$. For a field $k$ equipped with a discrete valuation $v$, we consider the valuation ring $\mathcal{O}_v$ of $v$ with maximal ideal $\mathfrak{m}_v$. We let $\kappa_v = \mathcal{O}_v / \mathfrak{m}_v$ be the residue field, and let $k_v$ be the $v$-adic completion of $k$. For a quadratic form $q$ defined over $k$, we write $q_v$ for the quadratic form $q_{k_v}$. Over the field $K = k(t)$, each monic irreducible polynomial $\pi \in k[t]$ induces a discrete valuation $v_{\pi}$ on $K$, and the completion and residue fields corresponding to $v_{\pi}$ will be denoted by $K_{\pi}$ and $\kappa_{\pi} \cong k[t] / (\pi)$, respectively.

Throughout the manuscript, we will repeatedly use Springer's Theorem on quadratic forms over complete discretely valued fields \cite[Proposition ~VI.1.9]{lam}, which we will refer to as Springer's Theorem. Springer's Theorem states that, over a field $K$ that is complete with respect to a discrete valuation $v$, with uniformizer $\pi$ and residue field $\kappa_v$ such that $\Char \kappa_v \ne 2$, a quadratic form $q \simeq q_1 \perp \pi \cdot q_2$ is anisotropic over $K$ if and only if both residue forms $\overline{q}_1$ and $\overline{q}_2$ are anisotropic over ~$\kappa_v$. 

Lastly, we make a small observation about Hasse principles over a field $k$ with respect to different sets of discrete valuations on $k$. Let $V$ and $W$ be two non-empty sets of non-trivial discrete valuations on $k$, and suppose that $V \subseteq W$. If quadratic forms over $k$ satisfy the Hasse principle (for isotropy/isometry) with respect to $V$, then they also satisfy the Hasse principle with respect to $W$. Equivalently, if a quadratic form $q$ over $k$ violates the Hasse principle with respect to $W$, then $q$ also violates the Hasse principle with respect to $V$.

\section{A small set of discrete valuations}
\label{transcendence degree 1}

Let us first focus on rational function fields in one variable over a field $k$ of characteristic $\ne 2$. Let $K = k(t)$, and let $\mathscr{P}$ be the set of monic irreducible polynomials in $k[t]$. Then if $V_{K/k}$ is the set of all discrete valuations on $K$ that are trivial on $k$, by \cite[Theorem ~2.1.4]{ep}, we know
\[
	V_{K/k} = V_{\mathscr{P}} \cup \left\{v_{\infty}\right\},
\]
where $V_{\mathscr{P}} = \left\{v_{\pi} \mid \pi \in \mathscr{P}\right\}$ and $v_{\infty}$ is the degree valuation with respect to $t$. Relative to the set of all discrete valuations on $K$, the set $V_{K/k}$ is small, but provides enough local data for the Hasse principle for isometry to hold, and for certain quadratic forms over $K$ to satisfy the Hasse principle for isotropy (see Proposition \ref{WHP over one-variable rational function field}). However, the main result of this section (Theorem \ref{SHP failure over A_i fields, weak assumptions}) shows that there are numerous counterexamples over $K$ to the Hasse principle for isotropy with respect to $V_{K/k}$ for certain ground fields $k$. 

Before proceeding, we recall that a quadratic form $q$ over a field $k$ is a \textit{Pfister neighbor} if ~$q$ is similar to a subform of a Pfister form $\varphi$ over $k$ with $\dim \varphi < 2 \dim q$ (see, e.g., \cite[Definition ~X.4.16]{lam}). In this situation, $\varphi$ is unique up to isometry by \cite[Proposition ~X.4.17]{lam}, and is called the \textit{Pfister form associated to ~$q$}. 

The assertions in the following proposition are well-known to experts, but the proofs do not seem to be written explicitly in the literature, so we include the ideas of the proofs for the sake of completeness.
\begin{prop}
\label{WHP over one-variable rational function field}
Let $k$ be any field of characteristic $\ne 2$, let $K = k(t)$, and let $V_{K/k}$ be the set of discrete valuations on $K$ that are trivial on $k$. Then 
\begin{enumerate}
	\item the Hasse principle for isometry holds over $K$ with respect to $V_{K/k}$, \label{part a}
	
	\item Pfister neighbors over $K$ satisfy the Hasse principle for isotropy with respect to $V_{K/k}$, \label{part b}
	
	\item Pfister forms over $K$, regular quadratic forms of dimension 2 or 3 over $K$, and regular four-dimensional quadratic forms over $K$ with trivial determinant satisfy the Hasse principle for isotropy with respect to $V_{K/k}$. \label{part c}
\end{enumerate}
\end{prop}
To prove Proposition \ref{WHP over one-variable rational function field}\ref{part a}, it suffices to consider $V_{\mathscr{P}} \subset V_{K/k}$, and the statement follows from using the Milnor exact sequence on Witt groups \cite[Theorem ~5.3]{milnor} together with the injectivity of the map $W(k) \to W(k((t)))$ on Witt groups (see, e.g., \cite[Exercise ~19.15]{ekm}).

Proposition \ref{WHP over one-variable rational function field}\ref{part b} then follows from Proposition \ref{WHP over one-variable rational function field}\ref{part a} by using the fact that a Pfister neighbor is isotropic if and only if its associated Pfister form is isotropic (see, e.g., \cite[Proof of Proposition ~X.4.17]{lam}), and the fact that a Pfister form is isotropic if and only if it is hyperbolic \cite[Theorem ~X.1.7]{lam}. Finally, Proposition \ref{WHP over one-variable rational function field}\ref{part c} is a particular case of Proposition \ref{WHP over one-variable rational function field}\ref{part b}, since all the quadratic forms considered in Proposition \ref{WHP over one-variable rational function field}\ref{part c} are Pfister neighbors by \cite[Examples X.4.18]{lam}.

\begin{remarks}
 \label{WHP remark} 
\begin{enumerate}
	\item For any field $k$ of characteristic $\ne 2$ equipped with a non-empty set $V$ of non-trivial discrete valuations with respect to which the $\WHP$ holds, the same ideas as above show that Pfister neighbors over $k$ satisfy the Hasse principle for isotropy with respect to $V$.
	
	\item For any $r \geq 1$, let $K_r = k(x_1, \ldots, x_r)$ be a purely transcendental field extension of transcendence degree $r$ over a field $k$ of characteristic $\ne 2$. Let $V_r$ be the set of discrete valuations on $K_r$ that are trivial on $K_{r-1}$ (here taking $K_0 = k$). Then $K_r \cong K_{r-1}(x_r)$, so with respect to $V_r$, the Hasse principle for isometry is satisfied, and Pfister neighbors over $K_r$ satisfy the Hasse principle for isotropy.
\end{enumerate}
\end{remarks}

The following result shows that, even when the $\WHP$ holds over purely transcendental field extensions of fields $\ell \in \mathscr{A}_i(2)$ for some $i$ (defined in Section \ref{notation}), the Hasse principle for isotropy can fail in several dimensions. 

\begin{theorem}
\label{SHP failure over A_i fields, weak assumptions}
Let $\ell$ be a field of characteristic $\ne 2$. Assume $\ell \in \mathscr{A}_i(2)$ for some $i \geq 0$ and $u\left(\ell\right) = 2^i$. For any $r \geq 1$ let $L_r = \ell(x_1, \ldots, x_r)$, and for $r \geq 2$ let $V_r$ be the set of discrete valuations on $L_r$ that are trivial on $L_{r-1}$. Then for $r \geq 2$ and any integer $m \ne 3$ such that 
\[
	2^{i + r - 1} < m \leq 2^{i+r},
\]
there exists an $m$-dimensional quadratic form over $L_r$ that violates the Hasse principle for isotropy with respect to $V_r$.
\end{theorem}
The proof is constructive and uses several lemmas which leverage both Springer's Theorem and the explicit description of $V_{K/k}$ for $K = k(t)$. 

\begin{lemma}
\label{anisotropic tensor product}
Let $q$ be an anisotropic quadratic form over a field $k$ of characteristic $\ne 2$. Then for any $r \geq 1$, the quadratic form $\langle \langle x_1, \ldots, x_r \rangle \rangle \otimes q$ is anisotropic over $k(x_1, \ldots, x_r)$. In particular, the Pfister form $\langle \langle x_1, \ldots, x_r \rangle \rangle$ is anisotropic over $k(x_1, \ldots, x_r)$.
\end{lemma}
\begin{proof} 
The second statement of the lemma follows from the first by taking $q = \langle 1 \rangle$, so it suffices to prove the first statement, which we do by inducting on $r \geq 1$.

First, suppose $r = 1$. Then by working over $k((x_1))$ and writing 
\[
	\langle \langle x_1 \rangle \rangle \otimes q = q \perp x_1 \cdot q,
\]
we see that both the first and second residue forms are equal to $q$, which is anisotropic over the residue field $k$ by assumption. So by Springer's Theorem, $\langle \langle x_1 \rangle \rangle \otimes q$ is anisotropic over $k((x_1))$, which contains $k(x_1)$, thus proving the base case.

Now suppose that for some $r \geq 1$, the form $\langle \langle x_1, \ldots, x_r \rangle \rangle \otimes q$ is anisotropic over $k(x_1, \ldots, x_r)$. Over $k(x_1, \ldots, x_r, x_{r+1})$, we can write $\langle \langle x_1, \ldots, x_{r+1} \rangle \rangle \otimes q$ as
\begin{align*}
	\langle \langle x_1, \ldots, x_r, x_{r+1} \rangle \rangle \otimes q &= \left(\langle \langle x_{r+1} \rangle \rangle \otimes \langle \langle x_1, \ldots, x_r \rangle \rangle\right) \otimes q = \langle \langle x_{r+1} \rangle \rangle \otimes \left(\langle \langle x_1, \ldots, x_r \rangle \rangle \otimes q \right).
\end{align*}
By the induction hypothesis, $\langle \langle x_1, \ldots, x_r \rangle \rangle \otimes q$ is anisotropic over $k(x_1, \ldots, x_r)$, so by the base case (with $k(x_1, \ldots, x_r)$ replacing $k$ and $x_{r+1}$ replacing $x_1$), $\langle \langle x_{r+1} \rangle \rangle \otimes \left(\langle \langle x_1, \ldots, x_r \rangle \rangle \otimes q\right)$ is anisotropic over $k(x_1, \ldots, x_r)(x_{r+1}) \cong k(x_1, \ldots, x_r, x_{r+1})$, completing the proof of the lemma by induction.
\end{proof}

Now, recall from Section \ref{notation} that for a quadratic form $q$ over $k$ that represents 1, $q'$ denotes the quadratic form over $k$ such that $q \simeq \langle 1 \rangle \perp q'$.
\begin{lemma}
\label{anisotropy in highest dimension}
Let $k$ be any field of characteristic $\ne 2$, and let $q$ be an anisotropic quadratic form over $k$ that represents 1. Then the quadratic form
\[
	\varphi = \langle x_2 + 1, -x_2 - x_1 \rangle \perp \langle 1, -x_2 \rangle \otimes q'  \perp x_1 \cdot \left(\langle \langle x_2 \rangle \rangle \otimes q \right)
\]
is anisotropic over $k(x_1, x_2)$.
\end{lemma}
\begin{proof}
We actually show that $\varphi$ is anisotropic over the field $k(x_2)((x_1))$ which contains $k(x_1, x_2)$. By Lemma \ref{anisotropic tensor product}, the second residue form of $\varphi$ is anisotropic over the residue field $k(x_2)$. So by Springer's Theorem, the lemma is proven if we show that the first residue form of $\varphi$,
\[
	\varphi_1 = \langle x_2 + 1, -x_2 \rangle \perp  \langle 1, - x_2 \rangle \otimes q',
\]
is anisotropic over $k(x_2)$. Rewrite $\varphi_1$ as
\[
	\varphi_1 = \langle x_2 + 1 \rangle \perp q' \perp -x_2 \cdot \left(\langle 1 \rangle \perp q' \right) \simeq  \left(\langle x_2 + 1\rangle \perp q'\right) \perp x_2 \cdot (-q)
\]
and consider $\varphi_1$ over $k((x_2))$. The first residue form of $\varphi_1$ over $k$ is $q$, and the second residue form of ~$\varphi_1$ is $-q$. By our choice of $q$, both residue forms of $\varphi_1$ are anisotropic over $k$. This implies that ~$\varphi_1$ is anisotropic over $k((x_2))$, which contains $k(x_2)$, and thus completes the proof of the lemma.
\end{proof}
\begin{remark}
If $q$ is any diagonal quadratic form over a field $k$, then scaling $q$ by its first entry results in a quadratic form over $k$ that represents 1. Moreover, if $q$ is anisotropic, then $q$ remains anisotropic after scaling by its first entry, and we can therefore apply Lemma \ref{anisotropy in highest dimension}.
\end{remark}

\begin{lemma}
\label{local isotropy of subforms}
Let $\ell$ be a field of characteristic $\ne 2$. Assume $\ell \in \mathscr{A}_i(2)$ for some $i \geq 0$ and $u(\ell) = 2^i$. Let $L_2 = \ell(x_1, x_2)$, and let $V_2$ be the set of discrete valuations on $L_2$ that are trivial on $L_1 = \ell(x_1)$. Let $q$ be an anisotropic $2^i$-dimensional quadratic form over $\ell$ that represents 1, and let ~$\varphi$ be the $2^{i + 2}$-dimensional quadratic form over $L_2$ defined by
\[
	\varphi = \langle x_2 + 1, -x_2 - x_1 \rangle \perp \langle 1, -x_2 \rangle \otimes q'  \perp x_1 \cdot \left(\langle \langle x_2 \rangle \rangle \otimes q \right).
\]
If $\psi$ is any subform of $\varphi$ such that $\dim \psi > 2^{i+1}$ and 
\[
	\langle x_2 + 1, -x_2 - x_1, x_1, x_1x_2 \rangle \subseteq \psi,
\] 
then $\psi$ is isotropic over $L_{2, v}$ for all $v \in V_2$.
\end{lemma}
\begin{proof} We prove the lemma by considering several cases for $v \in V_2$.

\underline{Case 1}: $v = v_{\infty}$ is the degree valuation with uniformizer $x_2^{-1}$.

The form $\psi$ contains the subform $\langle x_2 + 1, -x_2 - x_1 \rangle = x_2 \cdot \left\langle 1 + x_2^{-1}, -1 - x_1x_2^{-1} \right\rangle$. Scaling by ~$x_2^{-2}$, we have
\[
	\langle x_2 + 1, -x_2 - x_1 \rangle \simeq x_2^{-1} \cdot \left\langle 1 + x_2^{-1}, -1 - x_1x_2^{-1} \right\rangle,
\]
whose second residue form is $\langle 1, -1 \rangle$, which is isotropic. The second residue form of $\psi$ is therefore isotropic over the residue field $L_1$, hence $\psi$ is isotropic over $L_{2, v}$ by Springer's Theorem.
 
\underline{Case 2}: $v = v_{\pi}$, where $\pi = x_2$, $x_2 + 1$, or $x_2 + x_1$ is a divisor of at least one entry of $\psi$. 

The form $\psi$ contains the subforms $\langle -x_2 - x_1, x_1 \rangle, \langle x_1, x_1x_2 \rangle$, and $\langle x_2 + 1, x_1, x_1x_2 \rangle$, each of which reduces to an isotropic form over the respective residue field $\kappa_{\pi}$. So the first residue form of $\psi$ is isotropic over the residue field, hence $\psi$ is isotropic over $L_{2, \pi}$.

\underline{Case 3}: $v = v_{\pi}$, where $\pi \in L_1[x_2]$ is a monic irreducible polynomial different from $x_2$, $x_2 + 1$, and $x_2 + x_1$.

Let $m = \dim \psi$. In this case, each entry of $\psi$ is a unit in $\mathcal{O}_{v_{\pi}}$, so $\psi$ reduces to an $m$-dimensional quadratic form over the residue field $\kappa_{\pi}$. Since $\kappa_{\pi}$ is a finite extension of $L_1$, it satisfies property $\mathscr{A}_{i + 1}(2)$, thus $u\left(\kappa_{\pi}\right) \leq 2^{i+1} < m$ (see Section \ref{notation}). So the first residue form of $\psi$ is isotropic over $\kappa_{\pi}$, which implies that $\psi$ is isotropic over $L_{2, \pi}$. 

This covers all cases of $v \in V_2$, so the proof is complete.
\end{proof}
We can now prove Theorem \ref{SHP failure over A_i fields, weak assumptions}.
\begin{proof}[Proof of Theorem \ref{SHP failure over A_i fields, weak assumptions}]
We first observe that if $r > 2$, then $L_r = \ell(x_1, \ldots, x_r)$ is isomorphic to $\ell(x_1, \ldots, x_{r-2})(x_{r-1}, x_r)$. If $\ell \in \mathscr{A}_i(2)$, then $\ell(x_1, \ldots, x_{r-2}) \in \mathscr{A}_{i+ r - 2}(2)$, and by Lemma \ref{anisotropic tensor product}, if $\gamma$ is an anisotropic form over $\ell$, then the form $\langle \langle x_1, \ldots, x_{r-2} \rangle \rangle \otimes \gamma$ is anisotropic over $\ell(x_1, \ldots, x_{r-2})$. Hence $u\left(\ell(x_1, \ldots, x_{r-2})\right) = 2^{i + r - 2}$. It therefore suffices to prove the theorem for $r = 2$.

By assumption, $u(\ell) = 2^i$, so there exists a $2^i$-dimensional anisotropic quadratic form $q$ over $\ell$, which we can assume represents 1. By Lemmas \ref{anisotropy in highest dimension} and \ref{local isotropy of subforms}, if $\varphi$ is the $2^{i+2}$-dimensional form over $L_2$ defined by
\[
	\varphi = \langle x_2 + 1, -x_2 - x_1 \rangle \perp \langle 1, -x_2 \rangle \otimes q'  \perp x_1 \cdot \left(\langle \langle x_2 \rangle \rangle \otimes q \right),
\]
then any subform $\psi$ of $\varphi$ such that $m = \dim \psi > 2^{i+1}$ and $\langle x_2 + 1, -x_2 - x_1, x_1, x_1x_2 \rangle \subseteq \psi$, in particular $\varphi$ itself, violates the $\SHP$ over $L_2$ with respect to $V_2$.
\end{proof}

\begin{examples} The following are special cases of Theorem \ref{SHP failure over A_i fields, weak assumptions}. 

\begin{enumerate}
	\item For any prime $p \ne 2$, the field $\F_p \in \mathscr{A}_1(2)$ and $u\left(\F_p\right) = 2$. Then for any $\alpha \in \F_p^{\times} \setminus \F_p^{\times 2}$, the five-dimensional quadratic form over $\F_p(x_1, x_2)$ defined by
	\[
		\langle x_2 + 1, -x_2 - x_1, -\alpha, x_1, x_1x_2 \rangle
	\]
	violates the $\SHP$ with respect to $V_2$. 
	
	\item By \cite[Corollary ~2.7]{leep}, for any prime $p$, the field $\Q_p \in \mathscr{A}_2(2)$ and $u\left(\Q_p\right) = 4$. Let $u$ be a lift of a non-square in $\F_p^{\times}$ to $\Q_p$. Then the nine-dimensional quadratic form over $\Q_p(x_1, x_2)$ defined by
	\[
		\langle x_2 + 1, -x_2 - x_1, -p, -u, x_2u, x_1, x_1x_2, -x_1u, -x_1x_2u \rangle
	\]
	violates the $\SHP$ with respect to $V_2$.
\end{enumerate}
\end{examples}

\begin{remarks}
\begin{enumerate}
	\item If $i = 0$ and $r = 2$, the assumption that $m \ne 3$ in Theorem \ref{SHP failure over A_i fields, weak assumptions} is necessary. Indeed, by Proposition \ref{WHP over one-variable rational function field}\ref{part c}, three-dimensional quadratic forms over $L_2$ \textit{satisfy} the $\SHP$ with respect to $V_2$.
	
	\item In some instances, the assumption in Theorem \ref{SHP failure over A_i fields, weak assumptions} that $r \geq 2$ is necessary. For example, if $p \ne 2$ is a prime, then $\ell = \F_p \in \mathscr{A}_1(2)$, and the Hasse-Minkowski Theorem says that the $\SHP$ holds over $\F_p(x)$ with respect to all discrete valuations on $\F_p(x)$. Any discrete valuation on $\F_p$ is trivial, so the conclusion of Theorem \ref{SHP failure over A_i fields, weak assumptions} is false if $\ell = \F_p$ and $r = 1$.
\end{enumerate}
\end{remarks}

\section{Divisorial discrete valuations}
\label{main results}
Let $K/k$ be a finitely generated field extension of transcendence degree $r \geq 1$. A \textit{normal model} of $K/k$ is a normal $k$-variety $\mathscr{X}$ with function field $K$. A discrete valuation $v$ on $K$, trivial on $k$, is \textit{divisorial} if there exists some normal model $\mathscr{X}$ of $K/k$ and some prime divisor $D$ on $\mathscr{X}$ such that ~$v$ is equivalent to the discrete valuation on $K$ induced by $D$. Because $K$ has transcendence degree $r$ over $k$, if $v$ is a divisorial discrete valuation on $K$, then its residue field $\kappa_v$ is a finitely generated field extension of transcendence degree $r - 1$ over $k$.

Given a field $k$ equipped with a non-empty set $V$ of non-trivial discrete valuations, we say that ~$V$ \textit{satisfies the finite support property} if, given any $a \in k^{\times}$, the set
\[
	\{v \in V \bigm| v(a) \ne 0 \}
\]
is finite. Sets of discrete valuations that satisfy the finite support property arise naturally, and have also been considered in \cite{spinor groups, rap}. If $\mathscr{X}$ is a projective integral regular $k$-scheme with function field $K$, then by \cite[Lemma ~II.6.1]{hartshorne}, the set $V_{\mathscr{X}}$ of discrete valuations on $K$ induced by prime divisors on $\mathscr{X}$ satisfies the finite support property. We saw a particular example of this in Section ~\ref{transcendence degree 1} for $K = k(t)$: if $\mathscr{X} = \P^1_k$, then $V_{\mathscr{X}} = V_{K/k}$. Much like what we saw in Section ~\ref{transcendence degree 1}, for certain ground fields $k$ and $k$-varieties $\mathscr{X}$ with function field $K$, the Hasse principle for isometry over $K$ is satisfied with respect to $V_{\mathscr{X}}$ (Proposition \ref{trivial sha}), but numerous counterexamples exist over ~$K$ to the Hasse principle for isotropy with respect to $V_{\mathscr{X}}$ (Theorem \ref{main theorem}).

\subsection{The Hasse principle for isometry}

Let $k$ be any field of characteristic $\ne 2$, and for any $r \geq 1$, let $K_r = k(x_1, \ldots, x_r)$ be a purely transcendental field extension of transcendence degree ~$r$ over $k$. We saw in Section \ref{transcendence degree 1} that the Hasse principle for isometry holds over $K_r$ with respect to the set $V_r$ of discrete valuations on $K_r$ that are trivial on $K_{r-1}$ (here taking $K_0 = k$). Consequently, for any set $V$ of discrete valuations on $K_r$ that contains $V_r$, the $\WHP$ holds with respect to ~$V$; in particular, with respect to the set of all discrete valuations on $K_r$.

By \cite[Example ~VII.29.28]{kmrt}, we know that, given an $n$-dimensional quadratic form $q$ over a field $k$, the pointed Galois cohomology set $H^1(k, \Or_n(q))$ is in bijection with the set of isometry classes of $n$-dimensional quadratic forms over $k$, with (the isometry class of) $q$ being the distinguished element. By \cite[Example ~VII.29.29]{kmrt}, the pointed Galois cohomology set $H^1(k, \SO_n(q))$ is in bijection with the set of isometry classes of $n$-dimensional quadratic forms over $k$ with the same discriminant as $q$, again with $q$ being the distinguished element. If $W$ is a non-empty set of non-trivial discrete valuations on $k$, then a quadratic form $\varphi$ over $k$ is isometric to $q$ over $k_w$ for all $w \in W$ if and only if (the isometry class of) $\varphi$ belongs to the kernel of the global-to-local map
\begin{align*}
	H^1(k, \Or_n(q)) \to \prod_{w \in W} H^1(k_w, \Or_n(q)).
\end{align*}
The kernel of this global-to-local map gives a measure of the failure of the Hasse principle for isometry with respect to $W$. 

Let $K_r = k(x_1, \ldots, x_r)$ be as above, let $\mathscr{X}$ be a smooth projective integral model of $K_r/k$, and let $V_{\mathscr{X}}$ be the set of discrete valuations on $K_r$ induced by prime divisors on $\mathscr{X}$. For an $n$-dimensional quadratic form $q$ over $K_r$, let
\begin{align*} 
	\Sha_{\mathscr{X}}(K_r, \Or_n(q)) &= \ker\left(H^1(K_r, \Or_n(q)) \to \prod_{v \in V_{\mathscr{X}}} H^1(K_{r, v}, \Or_n(q))\right), \allowdisplaybreaks\\ 
	\Sha_{\mathscr{X}}(K_r, \SO_n(q)) &= \ker\left(H^1(K_r, \SO_n(q)) \to \prod_{v \in V_{\mathscr{X}}} H^1(K_{r, v}, \SO_n(q))\right), \allowdisplaybreaks \\ 
	\Sha^i_{\mathscr{X}}(K_r, \mu_2) &= \ker\left(H^i\left(K_r, \mu_2\right) \to \prod_{v \in V_{\mathscr{X}}} H^i(K_{r, v}, \mu_2)\right),\ i \geq 1.
\end{align*}
Since $\mu_2 = \{\pm 1 \}$ is contained in $K_r$, for any $j$ we can identify the Galois modules $\mu_2$ and $\mu_2^{\otimes j}$, which allows us to identify $H^i\left(K_r, \mu_2^{\otimes j}\right)$ and $H^i\left(K_r, \mu_2\right)$ for all $i$. 

For any discrete valuation $v$ on $K_r$ with residue characteristic $\ne 2$, we have well-defined residue homomorphisms (see \cite[II, \textsection 7]{gms})
\[
	\partial_v^i: H^i\left(K_r, \mu_2\right) \to H^{i-1}\left(\kappa_v, \mu_2\right).
\]
Let $V_{K_r/k}$ be the set of all discrete valuations on $K_r$ that are trivial on $k$. Then for any $v \in V_{K_r/k}$, since $\Char k \ne 2$, the residue field $\kappa_v$ has characteristic $\ne 2$ as well. Moreover, the set $V_{K_r/k}$ equals the set of discrete valuations on $K_r$ with residue characteristic $\ne 2$ whose valuation ring contains $k$, as this last condition forces invertible elements of $k$ to have valuation 0. For any $i \geq 1$, we consider the following unramified cohomology groups:
\begin{align*}
	H^i_{nr}\left(K_r/k, \mu_2\right) &= \bigcap_{v \in V_{K_r/k}} \ker \partial^i_v, \allowdisplaybreaks\\
	H^i\left(K_r, \mu_2\right)_{\mathscr{X}} &= \bigcap_{v \in V_{\mathscr{X}}} \ker \partial^i_v.
\end{align*}
Once again, the following result is well-known to experts, but does not seem to be written explicitly in the literature. We include a proof using unramified cohomology for the sake of completeness.
\begin{prop}
\label{trivial sha}
Let $k$ be any field of characteristic $\ne 2$, and for any $r \geq 1$ let $K_r = k(x_1, \ldots, x_r)$. Let $\mathscr{X}$ be a smooth projective integral model of $K_r/k$, and let $V_{\mathscr{X}}$ be the set of discrete valuations on $K_r$ induced by prime divisors on $\mathscr{X}$. Then for any $n$-dimensional quadratic form $q$ over $K_r$, the set $\Sha_{\mathscr{X}}(K_r, \emph{O}_n(q))$ is trivial; i.e., the Hasse principle for isometry holds over $K_r$ with respect to $V_{\mathscr{X}}$.
\end{prop}

\begin{remark}
If $V_r \subseteq V_{\mathscr{X}}$, this follows from Proposition \ref{WHP over one-variable rational function field}\ref{part a}, so there is nothing to prove. However, $V_{\mathscr{X}}$ does not necessarily contain $V_r$. For example, consider $\P^1_k \times \P^1_k$ with coordinates $x_1, x_2$. Let $P$ be the point given by $x_1 = x_2 = 0$, and let $L$ be the line $x_2 = 0$. By blowing up $\P^1_k \times \P^1_k$ at $P$, then blowing down the proper transform of $L$, we arrive at a smooth projective model $\mathscr{X}$ of $K_2/k$ such that the $x_2$-adic valuation, which belongs to $V_2$, is not contained in $V_{\mathscr{X}}$.
\end{remark}
 
\begin{proof}[Proof of Proposition \ref{trivial sha}]
The long exact cohomology sequence arising from the short exact sequence of groups $1 \to \SO_n(q) \to \Or_n(q) \to \mu_2 \to 1$ yields the 3-term short exact sequence
\begin{equation}
\label{exact sequence}
	\Sha_{\mathscr{X}}\left(K_r, \SO_n(q)\right) \rightarrow \Sha_{\mathscr{X}}\left(K_r, \Or_n(q)\right) \rightarrow \Sha^1_{\mathscr{X}}\left(K_r, \mu_2\right).
\end{equation}
Therefore, to prove Proposition \ref{trivial sha}, it suffices to show that the first and third terms of (\ref{exact sequence}) are trivial. For the third term of (\ref{exact sequence}), we show more: $\Sha^i_{\mathscr{X}}(K_r, \mu_2)$ is trivial for all $i \geq 1$.

For any $i \geq 1$, $\Sha^i_{\mathscr{X}}(K_r, \mu_2) \subseteq H^i\left(K_r, \mu_2\right)_{\mathscr{X}}$ by the definition of $\Sha^i_{\mathscr{X}}(K_r, \mu_2)$ and unramified cohomology. By \cite[Theorems ~4.1.1, 4.1.5]{colliot:santa_barbara}, for any $i \geq 1$, 
\[
	H^i\left(K_r, \mu_2\right)_{\mathscr{X}} \xrightarrow[]{\sim} H^i_{nr}\left(K_r/k, \mu_2\right) \xrightarrow[]{\sim} H^i(k, \mu_2).
\]
So we may view $\Sha^i_{\mathscr{X}}(K_r, \mu_2) \subseteq H^i(k, \mu_2)$. Since $\mathscr{X}$ is rational, there is a codimension one point $x$ on $\mathscr{X}$ whose induced discrete valuation $v_x$ has residue field isomorphic to ~$K_{r-1}$. The map $\psi_x: H^i(k, \mu_2) \to H^i(K_{r, v_x}, \mu_2)$ factors as
\[
	H^i(k, \mu_2) \xrightarrow[]{\sim} H^i_{nr}(K_{r-1}/k, \mu_2) \hookrightarrow H^i(K_{r-1}, \mu_2) \xrightarrow[]{\sim} H^i_{nr}(K_{r, v_x}, \mu_2) \hookrightarrow H^i(K_{r, v_x}, \mu_2).
\]
Here, the first isomorphism follows from \cite[Theorem 4.1.5]{colliot:santa_barbara}; the second and fourth maps are inclusions; and the third map is an isomorphism by the Gersten conjecture (see, e.g., \cite[pp. ~27]{colliot:santa_barbara}) and \cite[Theorem III.4.9]{artin}. The restriction of the injection $\psi_x$ to $\Sha^i_{\mathscr{X}}(K_r, \mu_2) \subseteq \ker \psi_x$ is trivial by the definition of $\Sha$. Hence $\Sha_{\mathscr{X}}^i(K_r, \mu_2)$ is trivial.

The triviality of $\Sha^i_{\mathscr{X}}(K_r, \mu_2)$ for all $i \geq 1$ then implies, by \cite[Theorem ~3.4]{spinor groups}, that $\Sha_{\mathscr{X}}(K_r, \SO_n(q))$ is trivial as well. This completes the proof of Proposition \ref{trivial sha}.	
\end{proof}

\begin{remark}
The statement of \cite[Theorem ~3.4]{spinor groups} assumed that $\dim q \geq 5$, but this assumption was not used in the proof.
\end{remark}

\subsection{The Hasse principle for isotropy}

In this section, we prove the main theorem:
\begin{theorem}
\label{main theorem}
Let $k$ be an algebraically closed field of characteristic $\ne 2$ that is not the algebraic closure of a finite field. Let $K$ be any finitely generated field extension of transcendence degree $r \geq 2$ over $k$, and let $V$ be any non-empty set of non-trivial divisorial discrete valuations on $K$ that satisfies the finite support property. Then for any integer $m \ne 3$ such that
\[
	2^{r-1} + 1 \leq m \leq 2^r,
\]
there is an $m$-dimensional quadratic form over $K$ that violates the Hasse principle for isotropy with respect to $V$.
\end{theorem}

\begin{example}
Let $K = k(x, y, z)$, where $k$ is an algebraically closed field of characteristic $\ne 2$ that is not the algebraic closure of a finite field, and let $\mathscr{X} = \P^3_{k}$. Since $V_{\mathscr{X}}$ contains the set of discrete valuations on $k(x, y, z)$ that are trivial on $k(x, y)$, by Proposition \ref{WHP over one-variable rational function field}\ref{part a}, the Hasse principle for isometry holds over $K$ with respect to $V_{\mathscr{X}}$. Moreover, by Proposition \ref{WHP over one-variable rational function field}\ref{part c}, quadratic forms over $K$ of dimensions two and three satisfy the $\SHP$ with respect to $V_{\mathscr{X}}$, as do four-dimensional quadratic forms over $K$ with trivial determinant. In particular, we see that in some instances the assumption in Theorem \ref{main theorem} that $m \ne 3$ is necessary. By Theorem \ref{main theorem}, there are counterexamples to the $\SHP$ over $K$ with respect to $V_{\mathscr{X}}$ in dimensions five through eight. Since $u(K) = 8$, quadratic forms of dimension ~$> 8$ over $K$ are isotropic, thus automatically satisfy the $\SHP$. The case of four-dimensional quadratic forms over $K$ with non-trivial determinant remains open.
\end{example}

Before proving Theorem \ref{main theorem}, we prove several results related to the Hasse principle for isotropy over finitely generated field extensions of fields $\ell \in \mathscr{A}_i(2)$ for some $i$ (defined in Section \ref{notation}). First we show that under certain assumptions, if quadratic forms of a particular dimension $m$ satisfy the Hasse principle for isotropy, then so do quadratic forms of dimension $\geq m$.
\begin{prop}
\label{new SHP from old over A_i fields}
Let $\ell$ be a field of characteristic $\ne 2$, and suppose that $\ell \in \mathscr{A}_i(2)$ for some $i \geq 0$. For any $r \geq 1$ such that $i + r > 2$, let $L$ be a finitely generated field extension of transcendence degree $r$ over $\ell$. Let $V$ be a non-empty set of non-trivial divisorial discrete valuations on $L$, trivial on $\ell$, that satisfies the finite support property. Then for any integer $m$ such that 
\[
	2^{i + r - 1} + 2 \leq m < 2^{i+r},
\] 
if regular $m$-dimensional quadratic forms over $L$ satisfy the Hasse principle for isotropy with respect to $V$, then so do regular $(m+1)$-dimensional quadratic forms over $L$.
\end{prop}
\begin{proof} This proof closely mirrors the proof of the Hasse-Minkowski Theorem for quadratic forms of dimension at least 5 found in \cite[pp. 172]{lam}.

Let $m$ be any integer such that $2^{i + r - 1} + 2 \leq m < 2^{i + r}$, and suppose regular $m$-dimensional quadratic forms over $L$ satisfy the $\SHP$ with respect to $V$. Let $q$ be an $(m + 1)$-dimensional quadratic form over $L$, and suppose that $q$ is isotropic over $L_v$ for all $v \in V$. Write $q = q_1 \perp q_2$, where $q_1 = \langle a_1, a_2 \rangle$, and $q_2 = \langle a_3, \ldots, a_{m + 1} \rangle$. Note that $\dim q_2 = m - 1 \geq 2^{i + r - 1} + 1$. 

Consider the following two disjoint subsets of $V$, whose union is $V$:
\begin{align*}
	S &= \left\{v \in V \bigm| q_{2, v} \text{ is isotropic over $L_v$} \right\}, \\
	T &= \left\{v \in V \bigm| q_{2, v} \text{ is anisotropic over $L_v$} \right\}.
\end{align*}
\underline{Claim}: $T$ is a finite set. 

Indeed, let $U \subseteq V$ be the subset defined by
\[
	U = \left\{v \in V \bigm| v(a_3) = v(a_4) = \cdots = v(a_{m+1}) = 0 \right\}.
\]
The set $V$ satisfies the finite support property, so $V \setminus U$ is a finite set. For any $v \in U$, each entry of ~$q_{2,v}$ is a unit in $\mathcal{O}_v$, so $q_{2,v}$ reduces to an $(m - 1)$-dimensional quadratic form over the residue field $\kappa_v$. Each such $v$ is a divisorial discrete valuation on $L$ that is trivial on $\ell$, so $\kappa_v$ is a finitely generated field extension of transcendence degree $r - 1$ over the field $\ell \in \mathscr{A}_i(2)$. Hence $u\left(\kappa_v\right) \leq 2^{i + r - 1} < m - 1$ (see Section ~\ref{notation}). Therefore, for any $v \in U$, $q_{2, v}$ is isotropic over $L_v$ by Springer's Theorem, hence $U \subseteq S$. This implies that $T = V \setminus S$ is contained in the finite set $V \setminus U$, proving the claim.

For any $v \in T$, because $q$ is isotropic over $L_v$, there exists some $z_v \in L_v^{\times}$ such that $z_v \in D_{L_v}(q_{1,v})$ and $-z_v \in D_{L_v}(q_{2,v})$. Thus, for any $v \in T$, we can write
\[
	z_v = a_1 x_v^2 + a_2 y_v^2 = q_{1,v}(x_v, y_v)
\]
for some $x_v$, $y_v \in L_v$. Since $T$ is a finite set, by Weak Approximation we can find $x$, $y \in L$ sufficiently close to $x_v$, $y_v$, respectively, for all $v \in T$, so that the element
\[
	q_1(x, y) = a_1x^2 + a_2y^2 =: z \in L
\]
is as close as desired to $z_v \ne 0$ for every $v \in T$. So $z \ne 0$, and $x$ and $y$ can be selected so that $z_v / z$ is close enough to $1$ in $L_v$ to guarantee that $z_v$ and $z$ belong to the same square class of $L_v$ for all $v \in T$. 

Here $z \in D_L(q_1)$, so we may write $q_1 \simeq \langle z, w \rangle$ for some $w \in L^{\times}$. Let $q^* = \langle z \rangle \perp q_2$, so that $q \simeq \langle w \rangle \perp q^*$. We next observe that the $m$-dimensional quadratic form $q^*$ is isotropic over $L_v$ for all $v \in V$. Indeed, if $v \in S$, then $q_{2,v}$ is isotropic over $L_v$, so $q^*$ must be isotropic over $L_v$ as well. For $v \in T$, since $z$ and $z_v$ belong to the same square class of $L_v$ and $-z_v \in D_{L_v} (q_{2,v})$, we see that $-z \in D_{L_v}(q_{2,v})$ for all $v \in T$. Therefore $q^* = \langle z \rangle \perp q_2$ is isotropic over $L_v$ for all $v \in T$. So $q^*$ is isotropic over $L_v$ for all $v \in S \cup T = V$, as asserted. By assumption, this implies that $q^*$ is isotropic over $L$. Thus $q \simeq \langle w \rangle \perp q^*$ is isotropic over $L$ as well, completing the proof.
\end{proof}

\begin{cor}
\label{most dimensions}
Let $\ell$ be a field of characteristic $\ne 2$ such that $\ell \in \mathscr{A}_i(2)$ for some $i \geq 0$, and let ~$L$ be a finitely generated field extension of transcendence degree $r \geq 1$ over $\ell$ such that $i + r \geq 2$. Let ~$V$ be a non-empty set of non-trivial divisorial discrete valuations on $L$, trivial on $\ell$, that satisfies the finite support property, and suppose there exists a $2^{i+r}$-dimensional quadratic form over $L$ that violates the Hasse principle for isotropy with respect to $V$. Then for any integer $m$ such that 
\[
	2^{i + r -1} + 2 \leq m \leq 2^{i + r},
\] 
there exists an $m$-dimensional quadratic form over $L$ that violates the Hasse principle for isotropy with respect to $V$.
\end{cor}
\begin{proof} If $i + r = 2$, the result is true by assumption since $2^{2-1} + 2 = 2^2$. So for $i + r > 2$, suppose by contradiction that the corollary is false. Let $m^*$ be the largest integer between $2^{i+r-1} + 2$ and $2^{i+r}$ such that $m^*$-dimensional quadratic forms over $L$ satisfy the $\SHP$ with respect to $V$. By assumption, $m^* < 2^{i+r}$. Applying Proposition \ref{new SHP from old over A_i fields}, since $m^*$-dimensional quadratic forms over $L$ satisfy the Hasse principle for isotropy with respect to $V$, then so do quadratic forms over $L$ of dimension $m^* + 1$. This contradicts the definition of $m^*$, and so for each $m$ such that $2^{i+r-1} + 2 \leq m < 2^{i+r}$, there must be an $m$-dimensional counterexample to the $\SHP$ over $L$ with respect to $V$.
\end{proof}

It remains to investigate the local-global behavior of quadratic forms with dimension $2^{n} + 1$ for some $n \geq 2$. Following \cite[Section 3]{auel-suresh}, over a field $k$ of characteristic $\ne 2$, given any elements $a_1, \ldots, a_n, d \in k^{\times}$, let $\langle \langle a_1, \ldots, a_n; d \rangle \rangle$ denote the $2^n$-dimensional quadratic form over $k$ obtained by multiplying the last entry, $a_1 \cdots a_n$, of the Pfister form $\langle \langle a_1, \ldots, a_n \rangle \rangle$ by $d$. For instance, if $n = 2$ we have
\[
	\langle \langle a_1, a_2 ; d \rangle \rangle = \langle 1, a_1, a_2, a_1a_2d \rangle.
\]
Such a form $\langle \langle a_1, \ldots, a_n; d \rangle \rangle$ is a \textit{twisted Pfister form} in the sense of Hoffmann \cite{hoffmann}.

As observed by Auel and Suresh \cite{auel-suresh}, by using a ``trick'' of Bogomolov, these twisted Pfister forms can be used to generate counterexamples to the Hasse principle for isotropy over function fields. Namely, let $k$ be an algebraically closed field of characteristic $\ne 2$. Let $K$ be a finitely generated field extension of transcendence degree $r \geq 1$ over $k$ and let $W$ be any non-empty set of non-trivial discrete valuations on $K$. According to Bogomolov's trick (see \cite[Proof of Theorem ~1.1]{bog}, \cite[Corollary ~1.2]{auel-suresh}), since $k$ is algebraically closed we can present $K$ as an odd degree extension of $k(x_1, \ldots, x_r)$ for some transcendence basis $x_1, \ldots, x_r$ of $K/k$. As such, any $w \in W$ restricts to a non-trivial discrete valuation on $k(x_1, \ldots, x_r)$, so let
\[
	V = \left\{\left.w\right|_{k(x_1, \ldots, x_r)} \bigm| w \in W\right\}.
\]
Suppose we have found a quadratic form $q$ over $k(x_1, \ldots, x_r)$ that violates the Hasse principle for isotropy with respect to $V$. Since $K$ is an odd degree extension of $k(x_1, \ldots, x_r)$, $q_K$ remains anisotropic over $K$ by Springer's Theorem on odd degree extensions \cite[Theorem ~VII.2.7]{lam}. For any $v \in V$, $q_v$ is isotropic over $k(x_1, \ldots, x_r)_v$, and since $v$ is the restriction of some $w \in W$, $k(x_1, \ldots, x_r)_v$ is contained in $K_w$. Hence $q_{K_w}$ is isotropic over $K_w$ for all $w \in W$, and therefore $q_K$ violates the $\SHP$ over $K$ with respect to $W$. In particular, these observations of Auel and Suresh \cite[Corollary ~1.2, Proposition ~1.3]{auel-suresh} prove the following:
\begin{lemma}[Auel-Suresh]
\label{reduction to rational function fields}
Let $K$ be any finitely generated field extension of transcendence degree $r \geq 1$ over an algebraically closed field $k$ of characteristic $\ne 2$. If there is an $n$-dimensional quadratic form over the rational function field $k(x_1, \ldots, x_r)$ that violates the Hasse principle for isotropy with respect to the set of all discrete valuations on $k(x_1, \ldots, x_r)$, then there is an $n$-dimensional quadratic form over $K$ that violates the Hasse principle for isotropy with respect to the set of all discrete valuations on $K$.
\end{lemma}

Now, for any algebraically closed field $k$ of characteristic $\ne 2$ that is not the algebraic closure of a finite field, let $k_0 \subset k$ be a subfield of $k$ equipped with a discrete valuation $v_0$ whose residue field has characteristic $\ne 2$. Then if $r \geq 2$ and we let
\[
	f_r = \prod_{i = 1}^r x_i(x_i - 1)(x_i - \lambda_i),
\]
where each $\lambda_i \in k_0 \setminus \{0, 1\}$ satisfies $v_0(\lambda_i) > 0$, \cite[Theorem ~4.1]{auel-suresh} states that over $k(x_1, \ldots, x_r)$, the $2^r$-dimensional twisted Pfister form $\langle \langle x_1, \ldots, x_r; f_r \rangle \rangle$ violates the Hasse principle for isotropy with respect to all discrete valuations on $k(x_1, \ldots, x_r)$. In particular, \cite[Theorem ~4.1]{auel-suresh}, together with Bogomolov's trick, proves the dimension $2^r$ case of Theorem \ref{main theorem}. We will use variants on the form $\langle \langle x_1, \ldots, x_r; f_r \rangle \rangle$ to prove the case of dimension $2^{r-1} + 1$.

\begin{lemma}
\label{anisotropic twisted Pfister form}
Let $k$ be any field of characteristic $\ne 2$. Then for any $r \geq 1$ the twisted Pfister form $\langle \langle x_1, \ldots, x_r; -1 \rangle \rangle$ is anisotropic over $k(x_1, \ldots, x_r)$.
\end{lemma}
\begin{proof}
Let $F = k(\sqrt{-1})$. Then over $F(x_1, \ldots, x_r)$, we have
\[
    \langle \langle x_1, \ldots, x_r; -1 \rangle \rangle_{F(x_1, \ldots, x_r)} \simeq \langle \langle x_1, \ldots, x_r \rangle \rangle.
\]
By Lemma \ref{anisotropic tensor product}, $\langle \langle x_1, \ldots, x_r \rangle \rangle$ is anisotropic over $F(x_1, \ldots, x_r)$ which contains $k(x_1, \ldots, x_r)$. So the form $\langle \langle x_1, \ldots, x_r; -1 \rangle \rangle$ must be anisotropic over $k(x_1, \ldots, x_r)$.
\end{proof}

\begin{lemma}
\label{global anisotropy}
Let $k$ be any field of characteristic $\ne 2$, and for any $r \geq 2$ let $K_r = k(x_1, \ldots, x_r)$. For $1 \leq i \leq r$, let $g_i(x_i) \in k[x_i]$ be polynomials of positive degree such that $g_i(0) \in k^{\times 2}$, and let $f_r = \prod_{i = 1}^r x_i g_i(x_i)$. Suppose the $2^r$-dimensional quadratic form $q_r = \langle \langle x_1, \ldots, x_r ; f_r \rangle \rangle$ is anisotropic over $K_r$. Then the $(2^r + 1)$-dimensional quadratic form
\[
	\widetilde{q}_r = q_r \perp \left\langle -x_{r+1}^2 - x_1 \cdots x_r \right\rangle
\]
is anisotropic over $K_{r+1} = k(x_1, \ldots, x_{r+1})$.
\end{lemma}
\begin{proof}
We first observe that
\[
	q_r = \langle \langle x_1, \ldots, x_r; f_r \rangle \rangle \simeq \left\langle 1, x_1, \ldots, x_r, \ldots, x_2 \cdots x_r, \prod_{i = 1}^r g_i(x_i) \right\rangle.
\]
Moreover, for each $i = 1, \ldots, r$, since $g_i(0) \ne 0$, $g_i(x_i)$ is a unit in $\mathcal{O}_{v_{x_i}}$, with reduction $g_i(0) \in \kappa_{x_i}$. The form $q_r$ is anisotropic over $K_r$, and we may write $q_r \simeq \langle 1 \rangle \perp q_r'$, so by \cite[Theorem ~IX.2.1]{lam}, 
\[
	x_1 \cdots x_r \in D_{K_r}\left(q_r'\right) \Longleftrightarrow x_{r+1}^2 + x_1 \cdots x_r \in D_{K_{r+1}}(q_r).
\]
\underline{Claim}: $x_1 \cdots x_r \not\in D_{K_r}\left(q_r'\right)$. 

The claim implies that $q_r$ does not represent $x_{r+1}^2 + x_1 \cdots x_r$ over $K_{r+1}$; or equivalently, the quadratic form
\[
	\widetilde{q}_r = q_r \perp \left\langle -x_{r+1}^2 - x_1 \cdots x_r \right\rangle 
\]
is anisotropic over $K_{r+1}$. It therefore suffices to prove the claim, which is equivalent to showing that the form $q_r' \perp \langle -x_1 \cdots x_r \rangle$ is anisotropic over $K_r$.

We prove the stronger claim, that $q_r' \perp \langle -x_1 \cdots x_r\rangle$ is anisotropic over the $x_1$-adic completion of ~$K_r$, which is $k(x_2, \ldots, x_r)((x_1))$, with residue field $k(x_2, \ldots, x_r)$. Since
\[
	q_r' \perp \langle -x_1 \cdots x_r \rangle \simeq \langle \langle x_2, \ldots, x_r \rangle \rangle' \perp \left\langle \prod_{i = 1}^r g_i(x_i) \right\rangle \perp x_1 \cdot \langle \langle x_2, \ldots, x_r; -1 \rangle \rangle,
\]
where $\langle \langle x_2, \ldots, x_r \rangle \rangle'$ is the pure subform of $\langle \langle x_2, \ldots, x_r \rangle \rangle$, the second residue form of $q_r' \perp \langle -x_1 \cdots x_r \rangle$ is the twisted Pfister form $\langle \langle x_2, \ldots, x_r; -1 \rangle \rangle$, which is anisotropic over $k(x_2, \ldots, x_r)$ by Lemma ~\ref{anisotropic twisted Pfister form}. By Springer's Theorem, to prove the claim it suffices to show that the first residue form of $q_r' \perp \langle -x_1 \cdots x_r \rangle$ is anisotropic over $k(x_2, \ldots, x_r)$. The first residue form is
\[
	\varphi_r := \langle \langle x_2, \ldots, x_r \rangle \rangle' \perp \left\langle g_1(0) \prod_{i = 2}^r g_i(x_i) \right\rangle
	\simeq \langle \langle x_2, \ldots, x_r \rangle \rangle' \perp \left\langle \prod_{i = 2}^r g_i(x_i) \right\rangle,
\]
where this last isometry follows because $g_1(0) \in k^{\times}$ is a square. We now prove, by induction on $r \geq 2$, that $\varphi_r$ is anisotropic over $k(x_2, \ldots, x_r)$.

First, suppose $r = 2$. Then 
\[
	\varphi_2 \simeq \langle\langle x_2 \rangle \rangle' \perp \langle g_2(x_2) \rangle = \langle g_2(x_2) \rangle \perp x_2 \cdot \langle 1 \rangle.
\]
Now consider $\varphi_2$ over $k((x_2))$. The first residue form of $\varphi_2$ is $\langle g_2(0) \rangle$, and the second residue form of $\varphi_2$ is $\langle 1 \rangle$. Both residue forms are anisotropic over $k$, so by Springer's Theorem, $\varphi_2$ is anisotropic over $k((x_2)) \supset k(x_2)$, proving the base case.

Now suppose that for some $r \geq 2$, $\varphi_r$ is anisotropic over $k(x_2, \ldots, x_r)$, and consider $\varphi_{r + 1}$ over $k(x_2, \ldots, x_r)((x_{r+1}))$, whose residue field is $k(x_2, \ldots, x_r)$. We have
\begin{align*}
	\varphi_{r + 1} &\simeq \langle \langle x_2, \ldots, x_r, x_{r + 1} \rangle \rangle' \perp \left\langle \prod_{i = 2}^{r+1} g_i(x_i) \right\rangle \\
	&= \langle\langle x_2, \ldots, x_r \rangle \rangle' \perp \left\langle \prod_{i = 2}^{r+1} g_i(x_i) \right\rangle \perp x_{r+1} \cdot \langle \langle x_2, \ldots, x_r \rangle \rangle.
\end{align*}
The second residue form of $\varphi_{r+1}$ is $\langle \langle x_2, \ldots, x_r \rangle \rangle$, which is anisotropic over $k(x_2, \ldots, x_r)$ by Lemma ~\ref{anisotropic tensor product}. Since $g_{r+1}(0) \in k^{\times}$ is a square, the first residue form of $\varphi_{r + 1}$ is
\begin{align*}
	\langle \langle x_2, \ldots, x_r \rangle \rangle' \perp \left\langle g_{r+1}(0) \prod_{i = 2}^r g_i(x_i) \right\rangle &\simeq \langle \langle x_2, \ldots, x_r \rangle\rangle' \perp \left\langle \prod_{i = 2}^r g_i(x_i) \right\rangle
	\simeq \varphi_r.
\end{align*}

By the induction hypothesis, $\varphi_r$ is anisotropic over $k(x_2, \ldots, x_r)$, so the first residue form of $\varphi_{r+1}$ is anisotropic. Both residue forms of $\varphi_{r+1}$ are anisotropic over $k(x_2, \ldots, x_r)$, so $\varphi_{r+1}$ is anisotropic over $k(x_2, \ldots, x_r)((x_{r+1}))$, which contains $k(x_2, \ldots, x_r, x_{r+1})$, completing the proof of the claim by induction, and the proof of the lemma as a whole.
\end{proof}

\begin{lemma}
\label{local isotropy}
Let $\ell$ be a field of characteristic $\ne 2$ such that $\ell \in \mathscr{A}_i(2)$ for some $i \geq 2$. Let $a_1, \ldots, a_i, d \in \ell^{\times}$ be elements such that $-a_1 \cdots a_i \not\in \ell^{\times 2}$ and the twisted Pfister form over $\ell$ defined by $q_i = \langle \langle a_1, \ldots, a_i; d \rangle \rangle$ is isotropic over $\ell_v$ for all discrete valuations $v$ on $\ell$. Then the $(2^i + 1)$-dimensional quadratic form $\widetilde{q}_i$ over $\ell(x)$ defined by
\[
	\widetilde{q}_i = q_i \perp \left\langle -x^2 - a_1 \cdots a_i \right\rangle
\]
is isotropic over $\ell(x)_w$ for all discrete valuations $w$ on $\ell(x)$.
\end{lemma}
\begin{proof}
We prove the lemma by considering several cases for the discrete valuation $w$ on $\ell(x)$. 

\underline{Case 1}: $w$ is non-trivial on $\ell$. 

In this case, if $v = \left.w\right|_{\ell}$, then $\ell_v$ is contained in $\ell(x)_w$, and $q_i$ is isotropic over $\ell_v$ by assumption. So $\widetilde{q}_i$ is isotropic over $\ell(x)_w$.
	
	The remaining cases cover the situation when $w$ is trivial on $\ell$.
	
\underline{Case 2}: $w = w_{\infty}$ is the degree valuation with respect to $x$. Thus, $\ell(x)_w = \ell\left(\left(x^{-1}\right)\right)$.

Multiplying the last entry of $\widetilde{q}_i$ by $x^{-2}$, we have
	\[
		\widetilde{q}_i \simeq q_i \perp \left\langle -1 - a_1\cdots a_i x^{-2} \right\rangle.
	\]
	Now $-1 - a_1 \cdots a_i x^{-2}$ is an $x^{-1}$-adic unit with reduction $-1$. Since $\langle 1 \rangle$ is a subform of $q_i$, the first residue form of $\widetilde{q}_i$ over $\kappa_w$ contains $\langle 1, -1 \rangle$, which is isotropic. Therefore $\widetilde{q}_i$ is isotropic over $\ell(x)_w$ by Springer's Theorem.
	
\underline{Case 3}: $w = w_{\pi}$ is the $\pi$-adic valuation for $\pi = x^2 + a_1 \cdots a_i$, which is irreducible since $-a_1 \cdots a_i \not\in ~\ell^{\times 2}$.

In this case, over the residue field $\kappa_{\pi} \cong \ell\left(\sqrt{-a_1 \cdots a_i}\right)$ we have
	\[
		a_1 = \frac{\left(\sqrt{-a_1 \cdots a_i}\right)^2}{-a_2 \cdots a_i}.
	\]
	The form $\widetilde{q}_i$ contains the subform $\langle a_1, a_2 \cdots a_i \rangle$, whose residue form mod $\pi$ is
	\[
		\overline{\langle a_1, a_2 \cdots a_i \rangle} = \left\langle \frac{\left(\sqrt{-a_1 \cdots a_i}\right)^2}{-a_2 \cdots a_i}, a_2 \cdots a_i \right\rangle \simeq \langle -a_2 \cdots a_i, a_2 \cdots a_i \rangle,
	\]
	which is isotropic over $\kappa_{\pi}$. Thus the first residue form of $\widetilde{q}_i$ is isotropic over $\kappa_{\pi}$, so $\widetilde{q}_i$ must be isotropic over $\ell(x)_{\pi}$.

\underline{Case 4}: $w = w_{\pi}$, where $\pi \in \ell[x]$ is any monic irreducible polynomial different from $x^2 + a_1 \cdots a_i$. 

In this case, each entry of $\widetilde{q}_i$ is a unit in $\mathcal{O}_{v_{\pi}}$, so $\widetilde{q}_i$ reduces to a $(2^i + 1)$-dimensional form over ~$\kappa_{\pi}$. The field $\kappa_{\pi}$ is a finite extension of $\ell \in \mathscr{A}_i(2)$, so $\kappa_{\pi} \in \mathscr{A}_i(2)$. Therefore $u\left(\kappa_{\pi}\right) \leq 2^i$. So the first residue form of $\widetilde{q}_i$ must be isotropic over $\kappa_{\pi}$, which implies that $\widetilde{q}_i$ is isotropic over $\ell(x)_{\pi}$.

These cases cover all possibilities for discrete valuations on $\ell(x)$, so the proof is complete.
\end{proof}

We can now prove Theorem \ref{main theorem}.
\begin{proof}[Proof of Theorem \ref{main theorem}]
By \cite[Theorem ~1]{auel-suresh}, there is a $2^r$-dimensional quadratic form over $K$ that violates the Hasse principle for isotropy with respect to $V$. This completes the proof if $r = 2$, so suppose $r \geq 3$. The field $k$ is algebraically closed, so $k \in \mathscr{A}_0(2)$. Moreover, any discrete valuation ~$v$ on $K$ is trivial on $k$ since any $x \in k^{\times}$ has $n$-th roots for all $n \in \mathbb{Z}$, so $v(x) \in \mathbb{Z}$ must be divisible by all $n \in \mathbb{Z}$, hence $v(x) = 0$. So by Corollary ~\ref{most dimensions}, for any $m$ such that $2^{r-1} + 2 \leq m \leq 2^r$, there is an $m$-dimensional quadratic form over $K$ that violates the Hasse principle for isotropy with respect to $V$. It therefore remains to find a quadratic form of dimension $2^{r-1} + 1$ over $K$ that violates the Hasse principle for isotropy with respect to $V$.

The set $V$ is contained in the set of all discrete valuations on $K$, so by Lemma \ref{reduction to rational function fields}, the proof will be complete if we can find a $(2^{r-1} + 1)$-dimensional quadratic form over the rational function field $k(x_1, \ldots, x_r)$ that violates the Hasse principle for isotropy with respect to all discrete valuations on $k(x_1, \ldots, x_r)$. Let $k_0 \subset k$ be a subfield with a discrete valuation ~$v_0$ with residue characteristic ~$\ne 2$, and for $1 \leq i \leq r - 1$, let $\lambda_i \in k_0 \setminus \{0, 1\}$ be elements such that $v_0(\lambda_i ) > 0$. If we let 
\[
	f_{r-1} = \prod_{i = 1}^{r-1} x_i(x_i - 1)(x_i - \lambda_i),
\]
then by \cite[Theorem ~4.1]{auel-suresh}, the quadratic form $q_{r-1} = ~\langle \langle x_1, \ldots, x_{r-1}; f_{r-1} \rangle \rangle$ over $k(x_1, \ldots, x_{r-1})$ violates the Hasse principle for isotropy with respect to the set of all discrete valuations on $k(x_1, \ldots, x_{r-1})$. That is, $q_{r-1}$ is anisotropic over $k(x_1, \ldots, x_{r-1})$, but is isotropic over the completion at each discrete valuation on that field. Because the field $k$ is algebraically closed, each ~$\lambda_i$ appearing in $f_{r-1}$ is a square in ~$k^{\times}$, so by Lemma \ref{global anisotropy}, the $(2^{r-1} +1)$-dimensional form over $k(x_1, \ldots, x_r)$ defined by 
\[
	\widetilde{q}_{r-1} = q_{r-1} \perp \left\langle -x_r^2 - x_1 \cdots x_{r-1} \right\rangle
\]
is anisotropic over $k(x_1, \ldots, x_r)$. The field $k \in \mathscr{A}_0(2)$, so $k(x_1, \ldots, x_{r-1}) \in ~\mathscr{A}_{r-1}(2)$ by \cite[Theorem ~2.3]{leep}. Therefore, by Lemma \ref{local isotropy} where $\ell = k(x_1, \ldots, x_{r-1})$, the form $\widetilde{q}_{r-1}$ is isotropic over $k(x_1, \ldots, x_r)_v$ for all discrete valuations $v$ on $k(x_1, \ldots, x_r)$. Thus $\widetilde{q}_{r-1}$ violates the Hasse principle for isotropy with respect to all discrete valuations on $k(x_1, \ldots, x_r)$, completing the proof.
\end{proof}

\begin{ack} The author would like to thank David Harbater, Julia Hartmann, Daniel Krashen, and Florian Pop for inspiring discussions and comments regarding the material of this article. The author would also like to thank an anonymous referee for helpful comments on an earlier version of the article. This paper is part of the author's Ph.D. thesis, currently being prepared under the supervision of David Harbater at the University of Pennsylvania.
\end{ack}

\providecommand{\bysame}{\leavevmode\hbox to3em{\hrulefill}\thinspace}
\providecommand{\href}[2]{#2}

\medskip

\noindent{\bf Author Information:}\\

\noindent Connor Cassady\\
Department of Mathematics, University of Pennsylvania, Philadelphia, PA 19104-6395, USA\\
email: cassadyc@sas.upenn.edu


\begin{thebibliography}{KMRT98}

\bibitem[Art62]{artin}
M.\ Artin,
\textit{Grothendieck Topologies},
Dept. of Mathematics, Harvard University, Cambridge, MA, 1962.

\bibitem[AS22]{auel-suresh}
A.\ Auel and V.\ Suresh,
\textit{Failure of the local-global principle for isotropy of quadratic forms over function fields},
Preprint, 2022, arXiv:1709.03707v2.

\bibitem[Bog95]{bog}
F.\ A.\ Bogomolov,
\textit{On the structure of Galois groups of the fields of rational functions},
{$K$}-theory and algebraic geometry: connections with quadratic forms
and division algebras ({S}anta {B}arbara, {CA}, 1992), 83--88, Proc.\ Sympos.\ Pure Math.\ \textbf{58}, Part 2, Amer.\ Math.\ Soc., Providence, RI, 1995.

\bibitem[CRR19]{spinor groups}
V.I.\ Chernousov, A.S.\ Rapinchuk, I.A.\ Rapinchuk,
\textit{Spinor groups with good reduction},
Compos. Math. \textbf{155} (2019), no. 3, 484--527.

\bibitem[CT95]{colliot:santa_barbara}
J.-L.~Colliot-Th\'el\`ene,
\textit{Birational invariants, purity and the {G}ersten conjecture},
{$K$}-theory and algebraic geometry: connections with quadratic forms
and division algebras ({S}anta {B}arbara, {CA}, 1992), 1--64, Proc.\ Sympos.\ Pure Math.\ \textbf{58}, Part 1, Amer.\ Math.\ Soc., Providence, RI, 1995.

\bibitem[EKM08]{ekm}
R.\ Elman, N.A.\ Karpenko, A.S.\ Merkurjev,
\textit{The algebraic and geometric theory of quadratic forms},
American Mathematical Society Colloquium Publications \textbf{56},
Amer. Math. Soc., Providence, RI, 2008.

\bibitem[EP05]{ep}
A.J.\ Engler, A.\ Prestel,
\textit{Valued fields},
Springer Monographs in Mathematics,
Springer-Verlag, Berlin, 2005.

\bibitem[GMS03]{gms}
S.\ Garibaldi, A.S.\ Merkurjev, J.-P.\ Serre,
\textit{Cohomological invariants in Galois cohomology},
University Lecture Series \textbf{28}, Amer. Math. Soc., Providence, RI, 2003.

\bibitem[Har77]{hartshorne}
R.\ Hartshorne,
\textit{Algebraic geometry},
Graduate Texts in Mathematics \textbf{52}, Springer-Verlag, New York-Heidelberg, 1977.

\bibitem[Hof96]{hoffmann}
D.W.\ Hoffmann,
\textit{Twisted Pfister forms},
Doc. Math. \textbf{1} (1996), no. 03, 67--102.

\bibitem[KMRT98]{kmrt}
M.-A.\ Knus, A.S.\ Merkurjev, M.\ Rost, and J.-P.\ Tignol,
\textit{The book of involutions},
American Mathematical Society Colloquium Publications \textbf{44},
Amer. Math. Soc., Providence, RI, 1998.

\bibitem[Lam05]{lam}
T.Y.\ Lam,
\textit{Introduction to quadratic forms over fields},
Graduate Studies in Mathematics \textbf{67},
Amer. Math. Soc., Providence, RI, 2005. 

\bibitem[Lee13]{leep}
D.\ Leep,
\textit{The $u$-invariant of $p$-adic function fields},
J.\ Reine Angew. Math. \textbf{679} (2013), 65--73.

\bibitem[Mil69]{milnor}
J.\ Milnor,
\textit{Algebraic $K$-theory and quadratic forms},
Invent. Math. \textbf{9} (1969/70), 318--344.


\bibitem[Pfi79]{pfister}
A.\ Pfister,
\textit{Systems of quadratic forms},
Colloque sur les Formes Quadratiques, 2 (Montpellier, 1977),
Bull. Soc. Math. France M\'{e}m. \textbf{59} (1979), 115--123.

\bibitem[RR22]{rap}
A.S.\ Rapinchuk, I.A.\ Rapinchuk,
\textit{Some finiteness results for algebraic groups and unramified cohomology over higher-dimensional fields},
J. Number Theory \textbf{233} (2022), 228--260.



\end{thebibliography}
\end{document}